\documentclass[11pt]{amsart}

\usepackage[T1]{fontenc}
\usepackage[utf8]{inputenc}
\usepackage{lmodern}
\usepackage{amsmath,amssymb,amsthm,mathtools}
\usepackage[a4paper,margin=1.15in]{geometry}
\usepackage{microtype}
\usepackage{enumitem}
\usepackage[colorlinks=true,linkcolor=blue,citecolor=blue,urlcolor=blue]{hyperref}

\newtheorem{theorem}{Theorem}[section]
\newtheorem{lemma}[theorem]{Lemma}
\newtheorem{corollary}[theorem]{Corollary}
\theoremstyle{definition}
\newtheorem{definition}[theorem]{Definition}
\newtheorem{remark}[theorem]{Remark}
\numberwithin{equation}{section}

\newcommand{\N}{\mathbb N}
\newcommand{\D}{\mathbb D}
\newcommand{\Cyl}{\operatorname{Cyl}}
\newcommand{\rng}{\operatorname{rng}}
\newcommand{\bfo}{\mathrm{bfo}}
\newcommand{\fo}{\mathrm{fo}}
\newcommand{\join}{\mathbin{\oplus}}

\title[Dense One-One Degrees in a Bounded Finite-One Degree]
{A Bounded Finite-One Degree Whose One-One Degrees\\
Form Exactly a Dense Linear Order}
\author{Patrizio Cintioli}
\address{Mathematics Division, School of Science and Technology, University of Camerino, Italy}
\email{patrizio.cintioli@unicam.it}



\begin{document}

\begin{abstract}
We construct a set $U\leq_T\emptyset''$ whose bounded finite-one degree,
ordered internally by one-one reducibility, consists exactly of a
countable dense linear order without endpoints. More precisely,

\[
  \left(
    \{[B]_1:B\equiv_{\bfo}U\},\leq_1
  \right)
  \cong
  (\mathbb Q,\leq).
\]

This exact realization contrasts with two previous results. In earlier
work, a bounded finite-one degree was constructed that contains a copy
of $(\mathbb Q,\leq)$, but the same degree also contains an infinite
antichain of one-one degrees and, more generally, embedded copies of
all countable partial orders; thus the dense chain does not exhaust the
degree. In a different direction, $m$-rigidity yields an almost-sure
and comeager obstruction: for a measure-$1$ and comeager class of sets,
the corresponding bounded finite-one degree contains an infinite
antichain of one-one degrees and hence is not linearly ordered.
The present construction shows that, despite this typical negative
behaviour, exact dense linear order can occur. In particular, it gives
an affirmative answer to Open Question~3 of Richter, Stephan, and Zhang.

The proof has two main parts. A block homogenization construction
produces a noncylindrical set $U$, a weak dyadic tower $(U_e)_{e\in\N}$,
computable reservoirs of both colours, and a base absorption property.
An abstract absorption-to-exhaustivity theorem then shows that every
member of the bounded finite-one degree of $U$ is one-one equivalent to
some finite autojoin $mU_e$, thereby yielding the exhaustive
classification above.
\end{abstract}

\maketitle

\section{Introduction}\label{sec:introduction}

We study the one-one degrees contained in a single bounded finite-one
degree. A bounded finite-one reduction is a total computable many-one
reduction whose fibres have a common finite bound. A bounded finite-one
degree may therefore split into several one-one degrees, and one may ask
which partially ordered structures arise when these degrees are ordered
by one-one reducibility.

Richter, Stephan, and Zhang \cite[Section~4]{RSZ2026} constructed a
nonrecursive bounded finite-one degree whose one-one degrees form exactly
a strictly ascending chain of order type $\omega$. Their work thus shows
that the full internal structure can be linearly ordered, rather than
merely contain a linearly ordered subfamily. In the conclusion of their
paper, they ask whether a dense linear order can also occur in this exact
sense \cite[Section~5, Open Question~3]{RSZ2026}:
\begin{quote}
\emph{Are there bounded finite-one degrees consisting exactly of a dense
linearly ordered set of one-one degrees?}
\end{quote}

Two previous results put this question into perspective. The first gives
a strong obstruction for typical sets. Recall that a set $A$ is
\emph{$m$-rigid} if every total computable many-one self-reduction of $A$
is eventually the identity. In \cite[Section~6]{CintioliRigidity}, it was
shown that the bounded finite-one degree of every $m$-rigid set contains
an infinite antichain of one-one degrees. Since the class of $m$-rigid
sets has Lebesgue measure $1$ and is comeager in Cantor space, it follows
that, for almost every set $A$, and for a comeager class of sets $A$, the
one-one degrees within $[A]_{\bfo}$ are not linearly ordered at all.
Richter, Stephan, and Zhang also observed the infinite-antichain
phenomenon for Martin-L\"of random sets \cite[Section~4]{RSZ2026}; the
$m$-rigidity argument extends it to a class that is also comeager.
Thus the configuration sought in Open Question~3 is ruled out both
almost everywhere and on a comeager class of sets.

The second result shows that density can nevertheless be realized as an
embedded suborder. In \cite{CintioliProfiles}, a block-density profile
method produced a noncomputable set $A\leq_T\emptyset'$ whose bounded
finite-one degree contains a copy of $(\mathbb Q,\leq)$. The same method
also yielded a single bounded finite-one degree containing both such a
dense chain and an infinite antichain, and, more generally, a degree into
which every countable partial order embeds. In particular, the degree
constructed there is not itself linearly ordered: the dense chain does
not exhaust its one-one degrees. These results establish the presence
of dense suborders, but do not settle the exact-realization problem.
Nor do they imply that every bounded finite-one degree containing a
dense chain must contain incomparable one-one degrees.

The present paper resolves this distinction by giving an affirmative
answer to Open Question~3. We construct a bounded finite-one degree in
which the dense chain exhausts the entire internal one-one degree
structure.

\medskip
\noindent
\textbf{Main Theorem (Corollary~\ref{cor:main}).}
\emph{There exists a noncomputable set $U\leq_T\emptyset''$ such that}
\[
  \left(
    \{[B]_1:B\equiv_{\bfo}U\},\leq_1
  \right)
  \cong
  (\D_{>0},\leq)
  \cong
  (\mathbb Q,\leq),
\]
\emph{where
$\D_{>0}=\{m/2^n:m\geq1,\ n\in\N\}$
is the set of positive dyadic rationals. In particular, the one-one
degrees in $[U]_{\bfo}$ form a countable dense linear order without
endpoints.}
\medskip

This affirmative existence result contrasts with the almost-sure and
comeager failure of the required configuration. Combining the theorem
with \cite[Section~6]{CintioliRigidity}, the class of sets whose bounded
finite-one degree has internal one-one order type $(\mathbb Q,\leq)$ is
nonempty, but null and meager. Here measure and category refer to sets
$A\in2^{\N}$ representing the degrees, not to a measure or topology
assumed on the quotient space of degrees.

More precisely, our construction supplies a sequence $U_0=U,U_1,\ldots$
with $U_n\equiv_1 2U_{n+1}$, where $mA$ denotes the join of $m$ copies of
$A$. Every set $B\equiv_{\bfo}U$ is one-one equivalent to some $mU_n$,
and the comparison relation is given exactly by
\[
  mU_n\leq_1 kU_\ell
  \quad\Longleftrightarrow\quad
  \frac{m}{2^n}\leq\frac{k}{2^\ell}.
\]
Thus positive dyadic multiplicities provide a classification of all the
one-one degrees in $[U]_{\bfo}$, not merely parameters for a selected
subfamily.

The proof is divided into two main parts. The Block Homogenization
Construction, Theorem~\ref{thm:block-homogenization}, builds a
noncylindrical set $U$ together with the sequence $(U_n)$, called a
\emph{weak dyadic tower}, and infinite computable subsets of both $U_n$
and its complement at each level. The construction uses nested
computable partitions into finite blocks. Homogenizing the number of
points from the range of a computable injection that fall into selected
blocks yields the \emph{base absorption property}
\[
  X\leq_1 U
  \quad\Longrightarrow\quad
  U\join X\equiv_1 mU_n
  \quad\text{for some }m\geq1\text{ and }n\in\N.
\]

The Absorption-to-Exhaustivity Theorem,
Theorem~\ref{thm:absorption-exhaustivity}, then proves abstractly that
these properties force every member of $[U]_{\bfo}$ to have the required
form. It combines a decomposition of bounded reductions into finitely
many injective pieces with a finite-one cancellation lemma:
\[
  P\join Z\leq_1 Q\join Z
  \quad\text{and}\quad
  Z\leq_{\fo}Q
  \quad\Longrightarrow\quad
  P\leq_1Q.
\]
Cancellation allows absorption to propagate along the tower and
ultimately removes the common summand from the comparison that
identifies an arbitrary $B\equiv_{\bfo}U$ with a positive dyadic
representative. This passage from absorption to an exhaustive
classification is what excludes additional one-one degrees outside
the dense chain.

Although $\emptyset''$ is used to select the construction data, every
reduction in the proof is an ordinary computable function. No computable
uniform family of reductions witnessing the tower equivalences is
assumed: each individual reduction uses only finitely many fixed levels.
The separation between the construction and the abstract classification
makes explicit the two tasks involved in exact realization: producing
the dyadic family and proving that it exhausts the degree.

\section{Preliminaries}

We use standard terminology and notation from computability theory.
General references are Rogers~\cite{Rogers1987},
Soare~\cite{Soare1987}, and the two volumes of
Odifreddi~\cite{Odifreddi1989,Odifreddi1999}.

All sets are subsets of $\N=\{0,1,2,\ldots\}$ and are identified with
their characteristic functions. We write
$\overline A=\N\setminus A$ and refer to the values $0$ and $1$ of a
characteristic function as the two membership colours. We write
$\rng(f)$ for the range of a function $f$.

We use ``c.e.'' for computably enumerable. Turing reducibility is
denoted by $\leq_T$, and $\emptyset'$ and $\emptyset''$ denote the
first and second Turing jumps of the empty set. The notation
$\Sigma^0_n$ and $\Pi^0_n$ refers to the usual arithmetical hierarchy.
Unless explicitly stated otherwise, all computability notions and all
reductions are unrelativized.

A total computable function $f$ is a many-one reduction from $A$ to $B$
if

\[
  x\in A \iff f(x)\in B
\]

for every $x$. It is a \emph{one-one reduction} if it is injective, a
\emph{finite-one reduction} if every fibre $f^{-1}(y)$ is finite, and a
\emph{bounded finite-one reduction} if there is a constant $k\geq1$,
possibly depending on $f$ but independent of $y$, such that

\[
  |f^{-1}(y)|\leq k
\]

for every $y$.

We write $\leq_1$, $\leq_{\fo}$ and $\leq_{\bfo}$ for the
corresponding reducibilities, and use $\equiv_1$, $\equiv_{\fo}$ and
$\equiv_{\bfo}$ for the induced equivalence relations. For
$\rho\in\{1,\fo,\bfo\}$, let
\[
  [A]_\rho=\{B\subseteq\N:B\equiv_\rho A\}.
\]
The equivalence classes induced by $\equiv_1$, $\equiv_{\fo}$ and
$\equiv_{\bfo}$ are called one-one degrees, finite-one degrees and
bounded finite-one degrees, respectively.

For every $m\geq1$, fix a computable bijection
\[
  \langle\cdot,\cdot\rangle_m:
  \{0,\ldots,m-1\}\times\N\longrightarrow\N
\]
with computable inverse. We suppress the subscript $m$ whenever
the relevant finite product is clear from the context. We also fix
a computable bijection $\N\times\N\to\N$ for the definition of
cylinders.

For $m\geq1$, the $m$-fold autojoin of $A$ is
\[
  mA=\{\langle i,x\rangle:i<m\ \&\ x\in A\}.
\]

The binary join is
\[
  A\join B
  =
  \{\langle0,x\rangle_2:x\in A\}
  \cup
  \{\langle1,x\rangle_2:x\in B\}.
\]
Up to a computable
permutation,
\[
  mA\join nA=(m+n)A.
\]
The cylinder over $A$ is
\[
  \Cyl(A)=\{\langle x,t\rangle:x\in A\ \&\ t\in\N\}.
\]
A set $A$ is a \emph{cylinder} if $A\equiv_1\Cyl(A)$.

We next introduce computable reservoirs, which allow computable
join summands to be absorbed without changing the one-one degree.
The idea is to map each reservoir injectively into part of itself,
leaving infinitely many unused target positions of each membership
colour for the additional summand.
Lemma~\ref{lem:reservoir-absorption} makes this precise and also
shows that the reservoir property is preserved under finite autojoins.

\begin{definition}\label{def:reservoir}
A set $A$ has \emph{computable reservoirs of both colours} if there are
computable infinite sets
\[
  R_1\subseteq A
  \qquad\text{and}\qquad
  R_0\subseteq\overline A.
\]
\end{definition}

\begin{lemma}[Reservoir absorption]\label{lem:reservoir-absorption}
Suppose that $A$ has computable reservoirs of both colours. Then:
\begin{enumerate}[label=\textup{(\roman*)}]
\item for every computable set $C$,
      \[
        A\join C\equiv_1 A;
      \]
\item $mA$ has computable reservoirs of both colours for every $m\geq 1$.
\end{enumerate}
\end{lemma}

\begin{proof}
Write
\[
  R_c=\{r_c(0)<r_c(1)<\cdots\}
  \qquad(c<2).
\]
Since each $R_c$ is computable and infinite, its increasing
enumeration $r_c$ is total computable. Moreover, for $x\in R_c$,
the unique index $j$ such that $x=r_c(j)$ is computable from $x$,
being equal to $|\{z<x:z\in R_c\}|$.

Define a computable colour-preserving injection $\rho:\N\to\N$ by
\[
  \rho(r_c(j))=r_c(2j)
\]
and $\rho(x)=x$ outside $R_0\cup R_1$. Its range omits the two
computable infinite sets
\[
  H_c=\{r_c(2j+1):j\in\N\}.
\]
Indeed, membership in $H_c$ is decidable: an element $x$ belongs
to $H_c$ exactly when $x\in R_c$ and its index in the increasing
enumeration of $R_c$ is odd.

Since $H_c\subseteq R_c$ for each $c<2$, we have
\[
  H_1\subseteq R_1\subseteq A,
  \qquad
  H_0\subseteq R_0\subseteq\overline A.
\]
Moreover,
\[
  \rng(\rho)=\N\setminus(H_0\cup H_1).
\]
Thus $H_1$ and $H_0$ are disjoint computable infinite sets of
free target positions of colours $1$ and $0$, respectively.

Map the $A$-summand of $A\join C$ by $\rho$. More precisely,
define a map $\theta:\N\to\N$ on the two summands of $A\join C$
as follows. For every $x\in\N$, put
\[
  \theta(\langle 0,x\rangle)=\rho(x).
\]

Map the $C$-summand into the free sets $H_0$ and $H_1$ according
to membership colour. More precisely, for $c<2$ and $x$ such
that $C(x)=c$, let
\[
  n_c(x)=|\{z<x:C(z)=c\}|
\]
be the rank of $x$ among the elements of $C$-colour $c$, and put
\[
  \theta(\langle 1,x\rangle)
  =r_c(2n_c(x)+1).
\]
Thus the copy of $x$ in the $C$-summand is mapped to the element
of $H_c$ having index $n_c(x)$ in the enumeration
\[
  H_c=\{r_c(2j+1):j\in\N\}.
\]
Since $C$ is computable, the map
\[
  x\longmapsto \bigl(C(x),n_{C(x)}(x)\bigr)
\]
is computable. Hence $\theta$ is total and computable.
Its restriction to the first summand is injective
because $\rho$ is injective. Its restriction to the second
summand is also injective: elements of the same $C$-colour have
distinct ranks, while elements of different colours are mapped
into the disjoint sets $H_0$ and $H_1$.

Moreover,
\[
  \rng(\rho)=\N\setminus(H_0\cup H_1),
\]
whereas the image of the second summand is contained in
$H_0\cup H_1$. Therefore the images of the two summands are
disjoint, and $\theta$ is injective on all of $\N$.

Finally, $\theta$ preserves membership colour. On the first
summand,
\[
  \langle 0,x\rangle\in A\join C
  \iff x\in A
  \iff \rho(x)\in A.
\]
On the second summand,
\[
  \langle 1,x\rangle\in A\join C
  \iff x\in C
  \iff C(x)=1
  \iff r_{C(x)}(2n_{C(x)}(x)+1)\in A,
\]
because $H_1\subseteq A$ and $H_0\subseteq\overline A$.
Consequently, $\theta$ is a one-one reduction of $A\join C$ to
$A$. Conversely, the map
\[
  x\longmapsto\langle 0,x\rangle
\]
is a one-one reduction of $A$ to $A\join C$. Hence
\[
  A\join C\equiv_1 A.
\]
For (ii), fix $m\geq1$ and, for each $c<2$, put
\[
  R_c^{(m)}
  =
  \{\langle 0,x\rangle:x\in R_c\}.
\]
Then $R_0^{(m)}$ and $R_1^{(m)}$ are computable and infinite.
Moreover,
\[
  R_1^{(m)}\subseteq mA
  \qquad\text{and}\qquad
  R_0^{(m)}\subseteq\overline{mA}.
\]
Thus $mA$ has computable reservoirs of both colours.

\end{proof}

The preceding absorption lemma will be used repeatedly in the following
slightly more flexible form. If a computable part of $B$ is a
computably coordinatized copy of $A$, while the remainder of $B$ is
computable, then the remainder can be absorbed without changing the
one-one degree.

\begin{lemma}[Computable remainder absorption]\label{lem:remainder}
Let $A$ have computable reservoirs of both colours. Let $B\subseteq\N$
and let $D\subseteq\N$ be computable. Suppose that there is a computable
bijection $\pi:\N\to D$ satisfying
\[
  x\in A \iff \pi(x)\in B,
\]
and that $B\cap\overline D$ is computable. Then $A\equiv_1 B$.
\end{lemma}

\begin{proof}
The map $\pi$ witnesses $A\leq_1 B$.

Put $K=B\cap\overline D$, which is computable by hypothesis.
Since $\pi$ is a computable bijection onto $D$, its inverse is
computable on $D$. Define
\[
   f(y)=
   \begin{cases}
      \langle 0,\pi^{-1}(y)\rangle_2, & y\in D,\\
      \langle 1,y\rangle_2, & y\notin D.
   \end{cases}
\]
The function $f$ is total computable because $D$ is computable.
It is injective on each of $D$ and $\overline D$, and the images
of these two sets lie in distinct columns.

If $y\in D$, then
\[
   y\in B
   \iff \pi^{-1}(y)\in A
   \iff f(y)\in A\join K.
\]
If $y\notin D$, then
\[
   y\in B
   \iff y\in K
   \iff f(y)\in A\join K.
\]
Thus $B\leq_1 A\join K$. By
Lemma~\ref{lem:reservoir-absorption}\textup{(i)},
\[
   A\join K\equiv_1 A.
\]
Consequently, $B\leq_1 A$, and hence $A\equiv_1 B$.
\end{proof}

The construction and the later classification argument will be organised
around two abstract properties. The first records a dyadic scaling
between successive representatives, while the second says that adjoining
any one-one subobject of the base can be absorbed into a finite autojoin
at some level of the tower. We isolate these notions now so that the
construction and the subsequent abstract argument can be kept separate.

\begin{definition}\label{def:weak-tower}
A sequence $(U_e)_{e\in\N}$ is a \emph{weak dyadic tower} over $U_0$ if
\[
  U_e\equiv_1 2U_{e+1}
  \qquad(e\in\N).
\]
No uniform procedure producing indices of these equivalences is assumed.
\end{definition}

\begin{definition}\label{def:BA}
Let $(U_e)_{e\in\N}$ be a weak dyadic tower over $U=U_0$. The pair
$(U,(U_e)_{e\in\N})$ has the \emph{base absorption property} if, for every
$X\leq_1U$, there are $e\in\N$ and $m\geq 1$ such that
\[
  U\join X\equiv_1 mU_e.
\]
\end{definition}

\section{The block homogenization construction}

We now construct the set $U$ and the weak dyadic tower that will satisfy
the properties isolated above. The construction uses computable
partitions into finite blocks and repeatedly reserves infinitely many
blocks on which a prescribed computable injection has the same
occupancy. The following constant-occupancy lemma is the combinatorial
device that makes this homogenization possible.

\begin{lemma}[Constant occupancy]\label{lem:constant-occupancy}
Let $F$ be a computable equivalence relation all of whose classes have the
same finite cardinality $d\geq 1$. Let
\[
  C_0,C_1,C_2,\ldots
\]
be the canonical enumeration of its classes in increasing order of their
least elements. Let $S\subseteq\N$ be computable and infinite, and let
$\varphi:\N\to\N$ be a total computable injection.
Then there is $k\in\{0,\ldots,d\}$ such that, for every prescribed
$r\geq 1$, there are pairwise disjoint computable infinite sets
\[
  P_0,\ldots,P_{r-1}\subseteq S
\]
such that
\[
  |\rng(\varphi)\cap C_j|=k
  \qquad
  (j\in P_0\cup\cdots\cup P_{r-1}).
\]
The sets $P_i$ may moreover be chosen to contain only indices of classes
whose least element exceeds any prescribed constant.
\end{lemma}

\begin{proof}
Let $m_j=\min C_j$. Since $F$ is computable, the set
\[
  M_F=\{m\in\N:(\forall z<m)\,\neg(z\,F\,m)\}
\]
of least representatives of the $F$-classes is computable. Hence
$j\mapsto m_j$, the increasing enumeration of $M_F$, is computable,
and
\[
  x\in C_j \iff x\,F\,m_j.
\]
Thus the family $(C_j)_{j\in\N}$ is uniformly computable. Moreover,
since every class has cardinality $d$, the complete finite list of
the elements of $C_j$ can be computed uniformly from $j$, an index
for $F$, and $d$.
For $t\leq d$, put
\[
  V_t=\{j\in S:|\rng(\varphi)\cap C_j|\geq t\}.
\]
For $t=0$, we have $V_0=S$. For $1\leq t\leq d$,
\[
  j\in V_t
  \iff
  j\in S\ \&\
  \exists n_0<\cdots<n_{t-1}\,
  \bigwedge_{i<t}\bigl(\varphi(n_i)\,F\,m_j\bigr).
\]
Since $\varphi$ is total computable and injective, this is a c.e.\
condition, uniformly in $t$.
Moreover,
\[
  V_0=S\supseteq V_1\supseteq\cdots\supseteq V_d.
\]
Let $k$ be the largest $t\leq d$ for which $V_t$ is infinite, and set
$V_{d+1}=\varnothing$.
Then $V_{k+1}$ is finite. Since $V_k$ is infinite, the set
\[
  E=V_k\setminus V_{k+1}
\]
is infinite. Moreover, for every $j\in S$,
\[
\begin{aligned}
  j\in E
  &\iff j\in V_k\ \&\ j\notin V_{k+1}\\
  &\iff |\rng(\varphi)\cap C_j|\geq k
     \ \&\
     |\rng(\varphi)\cap C_j|<k+1\\
  &\iff |\rng(\varphi)\cap C_j|=k.
\end{aligned}
\]
Here, when $k=d$, we use the convention $V_{d+1}=\varnothing$
together with the bound
\[
  |\rng(\varphi)\cap C_j|\leq d.
\]
Thus $E$ consists exactly of the indices $j\in S$ of classes
having occupancy $k$.

The choice of $k$ and the finite set $V_{k+1}$ is nonuniform.
Once that finite set has been fixed, however, $E$ is c.e.:
enumerate $V_k$ while omitting the finitely many forbidden elements.

Fix $r\geq 1$ and a prescribed bound $b\in\N$. The set
\[
   E_b=\{j\in E:m_j>b\}
\]
is c.e.\ and infinite, since $j\mapsto m_j$ is computable and
strictly increasing, and only finitely many indices satisfy
$m_j\leq b$.

Extract a total computable strictly increasing sequence
\[
   p(0)<p(1)<p(2)<\cdots
\]
of elements of $E_b$: choose an element when it is enumerated,
and after choosing $p(n)$ wait for an enumerated element greater
than $p(n)$ to define $p(n+1)$. Each search terminates because
$E_b$ is infinite.

For $i<r$, put
\[
   P_i=\{p(rn+i):n\in\N\}.
\]
For each $i<r$, the map $n\mapsto p(rn+i)$ is total computable
and strictly increasing, so its range $P_i$ is computable and
infinite. The sets $P_i$ are pairwise disjoint by injectivity
of $p$ and uniqueness of the remainder modulo $r$.
Finally, every $j\in P_0\cup\cdots\cup P_{r-1}$ belongs to $E_b$,
so its class has occupancy $k$ and satisfies $m_j>b$.

\end{proof}

\begin{remark}[Effectivity of the nonuniform choice]
\label{rem:hardcoding}
Given ordinary indices for $F$, $S$, and $\varphi$, together
with $d$, the sets $V_t$ are uniformly c.e. The property
``$V_t$ is infinite'' is $\Pi^0_2$, so $\emptyset''$ can choose
the largest suitable $k\in\{0,\ldots,d\}$.

If $k=d$, then $V_{k+1}=\varnothing$ by convention.
If $k<d$, the maximality of $k$ implies that $V_{k+1}$ is finite:
otherwise $k+1\leq d$ would be a larger index for which
$V_{k+1}$ is infinite.
In this case, put $W=V_{k+1}$ and fix a computable increasing
approximation $(W_s)_{s\in\N}$ by finite sets.

For each $s$, the condition
\[
   \forall u\geq s\;(W_u=W_s)
\]
is $\Pi^0_1$ and hence decidable in $\emptyset'$.
Since $W$ is finite, a search using $\emptyset'$ finds such
an $s$, and then $W=W_s$. Thus $\emptyset''$ suffices to obtain
the exact finite content of $V_{k+1}$.

Once $k$ and this finite set have been fixed, an ordinary
c.e.\ index for $E$ and, for any prescribed $r$ and bound $b$,
ordinary computable indices for the sets $P_i$ can be obtained
effectively from these data. The oracle is used only to select
indices and finite constants that are subsequently hardcoded
into ordinary programs; the resulting programs use no oracle.
\end{remark}

We now assemble the preceding ingredients into the main construction.
At each level, constant occupancy will provide the homogeneous block
families needed for absorption, while separate permanent colour
assignments diagonalize against possible cylinder embeddings. The same
block structure will simultaneously generate the dyadic tower and
preserve computable reservoirs of both colours.

\begin{theorem}[Block Homogenization Construction]
\label{thm:block-homogenization}
There exist a set $U\leq_T\emptyset''$ and a sequence $(U_e)_{e\in\N}$ such
that:
\begin{enumerate}[label=\textup{(\roman*)}]
\item $U_0=U$, and $(U_e)$ is a weak dyadic tower;
\item $U$ is noncomputable and is not a cylinder;
\item every $U_e$ has computable reservoirs of both colours;
\item $(U,(U_e))$ has the base absorption property. More precisely,
      there is an $\emptyset''$-computable list
      $(\varphi_e)_{e\in\N}$ consisting entirely of total computable
      injections from $\N$ to $\N$ and containing every such injection,
      possibly with repetitions, such that, for every $e\in\N$, there is
      $k_e\in\{0,\ldots,2^e\}$ satisfying
      \begin{equation}\label{eq:level-absorption}
        U\join\varphi_e^{-1}(U)
        \equiv_1(2^e+k_e)U_e.
      \end{equation}
\end{enumerate}
\end{theorem}

\begin{proof}
We use $\emptyset''$ only as a metatheoretic oracle.
Let $(\psi_i)_{i\in\N}$ be a fixed standard enumeration of the
partial computable functions from $\N$ to $\N$. Since the property
\[
  \psi_i\text{ is total and injective}
\]
is $\Pi^0_2$, the oracle $\emptyset''$ computes the increasing
enumeration
\[
  a_0<a_1<a_2<\cdots
\]
of all indices having this property. Put
\[
  \varphi_e=\psi_{a_e}.
\]
Thus every $\varphi_e$ is a total computable injection, and every
total computable injection occurs in the list, possibly more than
once extensionally.

Similarly, using a fixed standard enumeration of the partial
computable functions from $\N\times\N$ to $\N$, fix an
$\emptyset''$-computable list
\[
  \gamma_0,\gamma_1,\ldots
\]
consisting entirely of total computable injections and containing
every such injection. Repetitions are harmless.

We construct an increasing sequence
\[
  F_0\subseteq F_1\subseteq F_2\subseteq\cdots
\]
of equivalence relations. Each $F_e$ is an ordinary computable relation,
although the sequence of their indices need not be computable. Every
$F_e$-class will have cardinality $2^e$. At the beginning of level $e$,
each $F_e$-class carries an inherited computable label in $\{0,1,*\}$;
call this labelling $\lambda_e^-$. During level $e$ we extend the
commitments and obtain a completed computable labelling $\lambda_e^+$.
A class labelled $0$ or $1$ never changes colour, whereas a class labelled
$*$ is uncommitted. At the end of each level there are infinitely many
classes of each of the three labels.

Let
\[
  C^e_0,C^e_1,\ldots
\]
be the canonical enumeration of the $F_e$-classes by increasing least
element.
Since $F_e$ is computable and every $F_e$-class has the known
cardinality $2^e$, the class index of a number, the complete list of
the elements of a given class, and the least element of a class are
uniformly computable from $e$ and an index for $F_e$.
We maintain the following invariant. At the beginning of level $e$,
there are infinitely many $F_e$-classes labelled $*$, and every such
class has least element greater than $e-1$. Moreover, for $e>0$
there are infinitely many classes carrying each of the labels $0$ and
$1$. Classes carrying labels $0$ or $1$ have permanent final colour.
At the end of level $e$, the uncommitted classes will all have least
element greater than $e$.

At level $0$, let $F_0$ be equality and let $\lambda_0^-$ label every
singleton by $*$. Suppose that $F_e$ and $\lambda_e^-$ have been defined.

\medskip
\noindent
\textbf{Diagonalization against $\gamma_e$.}

Choose the least $F_e$-class $C$, in the canonical enumeration
$(C^e_j)_{j\in\N}$, such that $C$ is labelled $*$ and
$\min C>e$, and let
\[
   x=\min C.
\]
Since the values
\[
   \gamma_e(x,0),\gamma_e(x,1),\ldots
\]
are pairwise distinct and $C$ is finite, choose the least $t$ such
that
\[
   y=\gamma_e(x,t)\notin C.
\]
If the $F_e$-class of $y$ already has colour $c\in\{0,1\}$,
assign colour $1-c$ to $C$. If the $F_e$-class of $y$ is also
labelled $*$, assign colour $0$ to $C$ and colour $1$ to the
$F_e$-class of $y$. These assignments are permanent.

\medskip
\noindent
\textbf{Homogenization for $\varphi_e$.}
After the preceding finite action, let $S_e$ be the set of indices $j$ such
that $C^e_j$ is still labelled $*$ and has least element greater than $e$.
This is a computable infinite set. Apply
Lemma~\ref{lem:constant-occupancy} with $d=2^e$ and $r=4$. We obtain a
number $k_e\in\{0,\ldots,2^e\}$ and pairwise disjoint computable infinite sets
\[
  P_e,\quad Q^0_e,\quad Q^1_e,\quad R_e
  \subseteq S_e
\]
such that
\begin{equation}\label{eq:constant-occupancy-level}
  |\rng(\varphi_e)\cap C^e_j|=k_e
\end{equation}
for every $j$ in their union. Put
\begin{equation}
  H_e=P_e\cup Q^0_e\cup Q^1_e.
\end{equation}
Leave the classes indexed by $P_e$ labelled $*$, assign colour $0$ to the
classes indexed by $Q^0_e$, and assign colour $1$ to the classes indexed
by $Q^1_e$. Assign colour $0$ to every other class still labelled $*$,
including those indexed by $R_e$. Classes already coloured by earlier
levels or by the diagonalization retain their colours.

The resulting completed labelling is denoted by $\lambda_e^+$. Thus the
classes labelled $*$ by $\lambda_e^+$ are exactly those indexed by $P_e$,
and in particular their least elements exceed $e$. There are infinitely
many classes of all three labels: $P_e$ supplies $*$, $Q^0_e$ supplies
$0$, and $Q^1_e$ supplies $1$. The extra set $R_e$ ensures that
$\overline{H_e}$ is infinite.

\medskip
\noindent
\textbf{Passage to $F_{e+1}$.}
For each label $c\in\{0,1,*\}$, list in increasing order of their indices
all $F_e$-classes on which $\lambda_e^+$ has value $c$, pair consecutive
classes, and let each pair form one $F_{e+1}$-class. The new class
inherits the common label $c$; this defines $\lambda_{e+1}^-$. Then
$F_{e+1}$ is computable, each of its classes has size $2^{e+1}$, and every
$F_{e+1}$-class is the union of exactly two $F_e$-classes of the same
completed label.

To see computability explicitly, given $x,y$, compute the indices of their
$F_e$-classes and their $\lambda_e^+$-labels. If the labels differ, then
$x\not\mathrel{F_{e+1}}y$. If they agree, compute the ranks of the two
class indices in the computable list of classes having that label; the two
classes are amalgamated exactly when those ranks form the same consecutive
pair.
Since each label occurs on infinitely many $F_e$-classes, this
pairing produces infinitely many $F_{e+1}$-classes of each label.
Moreover, every $F_{e+1}$-class labelled $*$ has least element
greater than $e$, because both of its $F_e$-children do. Thus the
induction invariant holds at level $e+1$.
This completes the recursion.

\medskip
\noindent
\textbf{Effectivity of the recursion.}
We verify by induction on $e$ that $\emptyset''$ computes ordinary indices
for $F_e$, $\lambda_e^-$, $\lambda_e^+$, and the auxiliary computable
sets chosen at level $e$. Given the data for level $e$, the
diagonalization against $\gamma_e$ is an ordinary computable search. The
sets
\[
  V_t=\{j\in S_e:|\rng(\varphi_e)\cap C^e_j|\geq t\}
  \qquad(t\leq 2^e)
\]

are uniformly c.e. The oracle $\emptyset''$ decides which of them
are infinite and therefore chooses the maximal $k_e$. If
$k_e=2^e$, put $V_{k_e+1}=\varnothing$. Otherwise
$\emptyset''$ finds the exact finite content of $V_{k_e+1}$.

With that finite set hardcoded, the proof
of Lemma~\ref{lem:constant-occupancy} yields ordinary computable indices
for $P_e,Q^0_e,Q^1_e,R_e$. The completed labelling $\lambda_e^+$ and the
pairing construction above then yield ordinary computable indices for
$F_{e+1}$ and $\lambda_{e+1}^-$. Thus the entire finite construction up
to any prescribed level is uniformly recoverable from $\emptyset''$.

\medskip
\noindent
\textbf{Definition of $U$.}
Every number eventually belongs to a coloured class. Indeed, after level
$e$ every uncommitted $F_e$-class has least element greater than $e$; the
class containing $x$ has least element at most $x$, so it is coloured no
later than level $x$. Define
\begin{equation}
  x\in U
  \iff
  \text{the class containing $x$ is eventually assigned colour $1$}.
\end{equation}
Classes are amalgamated only with classes of the same label. Consequently
$U$ is saturated with respect to every $F_e$:
\begin{equation}
  x\,F_e\,y \quad\Longrightarrow\quad
  [x\in U\iff y\in U].
\end{equation}
By the preceding effectivity verification, $\emptyset''$ computes the
completed construction through level $x$. Since the colour of $x$ is
fixed by that level, it follows that $U\leq_T\emptyset''$.

\medskip
\noindent
\textbf{Noncylindricity.}
Let $\widehat g:\N\to\N$ be any total computable injection, and define
\[
   g(x,t)=\widehat g(\langle x,t\rangle).
\]
Then $g:\N\times\N\to\N$ is a total computable injection and hence
occurs as some $\gamma_e$.

At level $e$ we chose $x,t$ and permanently assigned
opposite colours to the $F_e$-classes of $x$ and $g(x,t)$. Thus
\[
  U(x)\neq U(g(x,t)).
\]

If $\widehat g$ were a one-one reduction of $\Cyl(U)$ to $U$, then
\[
  U(x)=\Cyl(U)(\langle x,t\rangle)
  =U(\widehat g(\langle x,t\rangle))
  =U(g(x,t)),
\]
contrary to the preceding inequality. Hence

\begin{equation}\label{eq:noncylindricity}
  \Cyl(U)\nleq_1U.
\end{equation}
The sets $Q^1_0$ and $Q^0_0$ are computable infinite subsets of $U$ and
$\overline U$, respectively. Thus $U$ is infinite and coinfinite. If
$U$ were computable, then both colour classes of $U$ and of $\Cyl(U)$
would be computable and infinite. Matching, for each colour, the $n$-th
element of the corresponding colour class of $\Cyl(U)$ with the $n$-th
element of that colour class of $U$ would give a computable
colour-preserving bijection, contradicting (\ref{eq:noncylindricity}). Therefore $U$ is
noncomputable and noncylindrical.

\medskip
\noindent
\textbf{The weak dyadic tower.}
Define
\begin{equation}
  U_e=\{j:C^e_j\subseteq U\}.
\end{equation}
Each $F_{e+1}$-class consists of exactly two $F_e$-classes of the same
final colour. Given an $F_e$-class, one can computably find its
$F_{e+1}$-parent and determine whether it is the first or second child;
conversely, from a parent and a child number one can recover the child.
This gives a computable colour-preserving bijection witnessing
\begin{equation}\label{eq:weak-dyadic-step}
  U_e\equiv_1 2U_{e+1}.
\end{equation}
Since $F_0$ is equality, $U_0=U$. Moreover, the computable sets $Q^1_e$
and $Q^0_e$ are infinite subsets of $U_e$ and $\overline{U_e}$,
respectively. Thus every $U_e$ has reservoirs of both colours.

For later use, let $d_e=2^e$. The map sending $x$ to the pair consisting
of the rank of $x$ in its $F_e$-class and the index of that class is a
computable colour-preserving bijection. Hence
\begin{equation}\label{eq:block-coordinate-equivalence}
  U\equiv_1 d_eU_e.
\end{equation}

\medskip
\noindent
\textbf{Base absorption.}
Fix $e$ and put
\[
  X_e=\varphi_e^{-1}(U),
  \qquad d=2^e.
\]
Let $h_e:\N\to H_e$ be the increasing enumeration of $H_e$, and set
\begin{equation}
  V_e=\{s:C^e_{h_e(s)}\subseteq U\}.
\end{equation}
The inverse images under $h_e$ of $Q^1_e$ and $Q^0_e$ are computable
infinite subsets of $V_e$ and $\overline{V_e}$, respectively, so $V_e$
has reservoirs of both colours.
Since the classes labelled $*$ by $\lambda_e^+$ are exactly those
indexed by $P_e\subseteq H_e$, every class whose index lies outside
$H_e$ already has its permanent final colour. Hence
\[
   U_e\cap\overline{H_e}
   =
   \{j\notin H_e:\lambda_e^+(j)=1\}
\]
is computable.

Apply
Lemma~\ref{lem:remainder} with
\[
  A=V_e,\qquad B=U_e,\qquad D=H_e,\qquad \pi=h_e.
\]
Since $s\in V_e$ iff $h_e(s)\in U_e$, the lemma gives
\begin{equation}\label{eq:core-equivalence}
  U_e\equiv_1V_e.
\end{equation}
Applying the witnessing one-one reductions separately on each of
the $d$ columns gives
\[
  dU_e\equiv_1dV_e.
\]
Together with (\ref{eq:block-coordinate-equivalence}),
\begin{equation}\label{eq:base-core-equivalence}
  U\equiv_1dV_e.
\end{equation}

Let
\begin{equation}
  L_e=\bigcup_{j\in H_e}C^e_j
  \qquad\text{and}\qquad
  D_e=\varphi_e^{-1}(L_e).
\end{equation}
Both sets are computable. If $k_e=0$, then $D_e=\varnothing$ by (\ref{eq:constant-occupancy-level}),
and membership in $X_e$ is computable because every value of $\varphi_e$
falls into an $F_e$-class outside $H_e$, whose colour is computable.
Lemma~\ref{lem:reservoir-absorption} and (\ref{eq:block-coordinate-equivalence}) then give
\[
  U\join X_e\equiv_1U\equiv_1dU_e,
\]
which is (\ref{eq:level-absorption}) in this case.

Suppose now that $k_e>0$. For $x\in D_e$, let $j(x)\in H_e$ be the index
of the $F_e$-class containing $\varphi_e(x)$, let
\[
  s(x)=h_e^{-1}(j(x)),
\]
and define
\begin{equation}
  r(x)=|\{z<x:\varphi_e(z)\in C^e_{j(x)}\}|.
\end{equation}
By (\ref{eq:constant-occupancy-level}), $r(x)<k_e$. The map
\begin{equation}\label{eq:rank-bijection}
   x\longmapsto (r(x),s(x))
\end{equation}
is a computable bijection from $D_e$ onto
$\{0,\ldots,k_e-1\}\times\N$. Its inverse is computable: given $(r,s)$,
search, in increasing order, for the $(r+1)$-st input whose
$\varphi_e$-value belongs to $C^e_{h_e(s)}$;
the search terminates because that class contains exactly
$k_e$ points of $\rng(\varphi_e)$. Furthermore,
\[
  x\in X_e
  \iff
  C^e_{j(x)}\subseteq U
  \iff
  s(x)\in V_e.
\]

Let $\beta_e:D_e\to\{0,\ldots,k_e-1\}\times\N$ denote the
bijection in (\ref{eq:rank-bijection}). Define
$\pi_e:\N\to D_e$ by
\[
   \pi_e(\langle r,s\rangle_{k_e})
   =
   \beta_e^{-1}(r,s)
   \qquad(r<k_e,\ s\in\N).
\]
Then $\pi_e$ is a computable bijection.

The preceding colour
calculation says exactly that
\[
  z\in k_eV_e \iff \pi_e(z)\in X_e.
\]
On $\overline{D_e}$, membership in $X_e$ is computable, because
$\varphi_e(x)$ lies in an $F_e$-class outside $H_e$. Since $k_eV_e$ has
reservoirs, Lemma~\ref{lem:remainder}, applied with
\[
  A=k_eV_e,\qquad B=X_e,\qquad D=D_e,\qquad \pi=\pi_e,
\]
yields
\begin{equation}\label{eq:pullback-core-equivalence}
  X_e\equiv_1k_eV_e.
\end{equation}
Using (\ref{eq:base-core-equivalence}), (\ref{eq:pullback-core-equivalence}), and (\ref{eq:core-equivalence}), we obtain
\[
  U\join X_e
  \equiv_1(d+k_e)V_e
  \equiv_1(d+k_e)U_e.
\]
This proves (\ref{eq:level-absorption}) for the fixed $e$.

All reductions used above are ordinary computable functions. For the
fixed level $e$, their programs may hardcode the finite constant $k_e$
and ordinary indices for the relevant computable level-$e$ data,
including $F_e$, $H_e$, $\lambda_e^+$, $\varphi_e$, $h_e$, and the
auxiliary computable maps. Later levels do not alter any of these
data, do not split any $F_e$-class, and do not change any colour
already assigned at level $e$. They may amalgamate the still
uncommitted $F_e$-classes indexed by $P_e$, but their only effect
relevant to the preceding colour calculations is to determine the
final colours of those classes. These final colours are exactly
recorded by the core set $V_e$. Hence all the reductions constructed
above remain valid for the final sets.

Since every total computable injection occurs among the $\varphi_e$,
the base absorption property follows.
\end{proof}

\begin{remark}\label{rem:nonuniform}
Every fixed level yields ordinary computable witnesses.
No computable uniform choice of witnesses throughout the whole
infinite tower is assumed. Each witnessing reduction constructed
above uses only finitely many fixed levels, whose ordinary
computable indices may be hardcoded into its program.

The levels are fixed for the reduction as a whole, not merely
finite in number on each individual input.
\end{remark}

\section{From absorption to exhaustivity}

We now leave the specific block construction and turn to the abstract
reduction-theoretic argument that converts base absorption into an
exhaustive classification of the whole bounded finite-one degree.
The main ingredients will be a cancellation principle for one-one
reductions, the strict behaviour of finite autojoins of a noncylindrical
set, and the resulting dyadic scalar calculus.

The first ingredient is a cancellation lemma. In general, a common
summand cannot simply be cancelled from a one-one reduction. The
additional hypothesis $Z\leq_{\fo}Q$ provides enough finite control to
make cancellation possible. The proof encodes the two reductions in a
computable graph and then extracts, component by component, an injective
reduction from $P$ to $Q$.

\begin{lemma}[Finite-one cancellation]\label{lem:cancellation}
Let $P,Q,Z\subseteq\N$. If
\[
  P\join Z\leq_1Q\join Z
  \qquad\text{and}\qquad
  Z\leq_{\fo}Q,
\]
then
\[
  P\leq_1Q.
\]
\end{lemma}

\begin{proof}
Fix an injective reduction
\[
  f:P\join Z\leq_1Q\join Z
\]
and a finite-one reduction
\[
  g:Z\leq_{\fo}Q.
\]
Use four disjoint computable copies of $\N$, denoted by
\[
  P^-,\quad Z^-,\quad Q^+,\quad Z^+.
\]
Write $x_P^-,z_Z^-,q_Q^+,w_Z^+$ for the vertices carrying the indicated
underlying numbers, and give them membership colours $P(x)$, $Z(z)$,
$Q(q)$, and $Z(w)$, respectively. We use the fixed binary coding of the
two summands, so every value of $f$ has a unique form
$\langle j,y\rangle$ with $j<2$.

Build a computable, and hence c.e., undirected graph as follows. For each $x$, if
\[
  f(\langle 0,x\rangle)=\langle 0,q\rangle,
\]
add an edge from $x_P^-$ to $q_Q^+$, while if
\[
  f(\langle 0,x\rangle)=\langle 1,z\rangle,
\]
add an edge from $x_P^-$ to $z_Z^+$. Similarly, for each $z$, use the
value of $f(\langle 1,z\rangle)$ to join $z_Z^-$ to the corresponding
vertex in $Q^+\cup Z^+$. In addition, add
\begin{enumerate}[label=\textup{(\arabic*)}]
\item an edge between $z_Z^-$ and $z_Z^+$ for every $z$;
\item an edge between $z_Z^+$ and $g(z)_Q^+$ for every $z$.
\end{enumerate}
Adjacency is decidable from the types and underlying numbers of
the two vertices, by evaluating $f$ or $g$ and checking the
matching condition where appropriate. No membership test in
$P$, $Q$, or $Z$ is used in constructing the graph.
Every edge joins vertices of the same membership colour: this follows from
the reduction property of $f$ for the first family of edges, from the
identical $Z$-labels for the matching edges, and from the reduction
property of $g$ for the last family. Hence each connected component is
monochromatic.

Let $K$ be a finite component. The matching between $Z^-$ and $Z^+$
gives
\[
  |K\cap Z^-|=|K\cap Z^+|.
\]
Since $f$ is injective and maps every left vertex of $K$ to a right vertex
of $K$,
\[
  |K\cap P^-|+|K\cap Z^-|
  \leq
  |K\cap Q^+|+|K\cap Z^+|.
\]
Hence
\begin{equation}\label{eq:finite-component-balance}
  |K\cap P^-|\leq|K\cap Q^+|.
\end{equation}

Every infinite component contains infinitely many $Q^+$-vertices.
Indeed, let $K$ be a component such that $K\cap Q^+$ is finite.
If $z_Z^+\in K$, then the edge from $z_Z^+$ to $g(z)_Q^+$
shows that $g(z)_Q^+\in K$. Hence
\[
  \{z:z_Z^+\in K\}
  \subseteq
  \bigcup_{\{q:q_Q^+\in K\}} g^{-1}(q).
\]
The set on the right is finite, since $K\cap Q^+$ is finite and
$g$ is finite-one. Thus $K\cap Z^+$ is finite. The matching
edges then show that $K\cap Z^-$ is finite. Finally, $f$ maps
the left-hand vertices of $K$ injectively into the finite set
\[
  K\cap(Q^+\cup Z^+),
\]
so $K\cap P^-$ is finite as well. Therefore $K$ is finite.
This proves the claim.

Construct $h:\N\to\N$ recursively in increasing order of its
argument. The connected component of a vertex is uniformly c.e.,
since it can be enumerated by enumerating all finite paths from that
vertex.

Suppose that $h(0),\ldots,h(p-1)$ have been defined. Enumerate the
component of $p_P^-$ until a vertex $q_Q^+$ is found such that
\[
  q\notin\{h(0),\ldots,h(p-1)\},
\]
and put $h(p)=q$.

This procedure does not require deciding whether the component
is finite or infinite, nor computing the complete content of
any fibre of $g$. It uses only the enumeration of the component
and the finite list of previously chosen values.
This search terminates. If the component is infinite, it contains
infinitely many $Q^+$-vertices, whereas only finitely many values
have previously been used. Suppose instead that the component is
a finite component $K$. Any previously used $Q^+$-vertex belonging
to $K$ was chosen for an earlier argument $p'$ with
$p_P^{\prime-}\in K$. Since the current vertex $p_P^-$ also lies
in $K$, at most
\[
  |K\cap P^-|-1
\]
vertices of $K\cap Q^+$ have previously been used. By
(\ref{eq:finite-component-balance}),
\[
  |K\cap P^-|\leq |K\cap Q^+|,
\]
so an unused $Q^+$-vertex remains.

It follows that $h$ is total and computable. It is injective by
construction. Since $p_P^-$ and $h(p)_Q^+$ belong to the same
monochromatic component,
\[
  p\in P\iff h(p)\in Q.
\]
Thus $P\leq_1Q$.
\end{proof}

Cancellation allows us to convert noncylindricity into a strict
separation between finite autojoins. The next lemma shows that, for a
noncylindrical set, adding further copies can never collapse the
one-one degree.

\begin{lemma}[Strict finite autojoins]\label{lem:strict-autojoins}
If $A$ is not a cylinder, then
\[
  A<_12A<_13A<_1\cdots.
\]
\end{lemma}

\begin{proof}
First note that $1A\equiv_1 A$, via the computable bijection
$x\mapsto\langle 0,x\rangle_1$.

For every $m\geq1$, the map
\[
  \langle i,x\rangle_m
  \longmapsto
  \langle i,x\rangle_{m+1}
\]
is a one-one reduction of $mA$ to $(m+1)A$.

Suppose first that
\[
  2A\leq_1A,
\]
and let $f$ be an injective reduction witnessing this. For $i<2$,
define
\[
  a_i(x)=f(\langle i,x\rangle_2).
\]
Then $a_0,a_1:\N\to\N$ are total computable colour-preserving
injections. Their ranges are disjoint, because $f$ is injective and
the two columns of $2A$ are disjoint.
For each $n\in\N$, put
\[
  e_n=a_1^n\circ a_0,
\]
where $a_1^0=\mathrm{id}_{\N}$.

Each $e_n$ is a total computable colour-preserving injection, uniformly
in $n$. The ranges of the maps $e_n$ are pairwise disjoint. Indeed,
suppose that $m<n$ and
\[
  e_m(x)=e_n(y).
\]
Then
\[
  a_1^m(a_0(x))
  =
  a_1^m\bigl(a_1^{n-m}(a_0(y))\bigr).
\]
Since $a_1^m$ is injective,
\[
  a_0(x)=a_1^{n-m}(a_0(y)).
\]
The left-hand side belongs to $\rng(a_0)$, while the right-hand side
belongs to $\rng(a_1)$ because $n-m\geq1$. This contradicts the
disjointness of the two ranges.

It follows that the map
\[
  \langle x,n\rangle
  \longmapsto
  e_n(x)
\]
is a total computable injection. Moreover,
\[
  \langle x,n\rangle\in\Cyl(A)
  \iff x\in A
  \iff e_n(x)\in A,
\]
so it is a one-one reduction of $\Cyl(A)$ to $A$. Conversely, the
map
\[
  x\longmapsto\langle x,0\rangle
\]
is always a one-one reduction of $A$ to $\Cyl(A)$. Hence
\[
  A\equiv_1\Cyl(A),
\]
contrary to the hypothesis.

Now suppose that
\[
  (m+1)A\leq_1mA
\]
for some $m\geq2$. Put $Z=(m-1)A$. Up to computable permutations
of the columns,
\[
  (m+1)A\equiv_1 2A\join Z,
  \qquad
  mA\equiv_1 A\join Z.
\]
Thus the assumed reduction yields
\[
  2A\join Z\leq_1A\join Z.
\]
The projection
\[
  \langle i,x\rangle_{m-1}\longmapsto x
\]
is a bounded finite-one reduction of $Z$ to $A$, and in particular
\[
  Z\leq_{\fo}A.
\]
Lemma~\ref{lem:cancellation}, applied with
\[
  P=2A,\qquad Q=A,\qquad Z=(m-1)A,
\]
therefore gives
\[
  2A\leq_1A,
\]
which, by the first part of the proof, would make $A$ a cylinder.
This is again a contradiction.

Consequently,
\[
  (m+1)A\nleq_1mA
\]
for every $m\geq1$. Since $mA\leq_1(m+1)A$, we conclude that
\[
  A<_12A<_13A<_1\cdots.
\]
\end{proof}

\begin{lemma}[Transfer of reservoirs]\label{lem:transfer-reservoirs}
If $A$ has computable reservoirs of both colours and
$A\leq_{\fo}B$, then $B$ has computable reservoirs of both colours.
\end{lemma}

\begin{proof}
Let
\[
  R_1\subseteq A
  \qquad\text{and}\qquad
  R_0\subseteq\overline A
\]
be computable infinite reservoirs, and let
\[
  f:A\leq_{\fo}B.
\]
For $c<2$, put
\[
  W_c=f[R_c].
\]
Since $R_c$ is computable and $f$ is total computable, $W_c$ is
c.e. Moreover, the reduction property of $f$ gives
\[
  W_1\subseteq B
  \qquad\text{and}\qquad
  W_0\subseteq\overline B.
\]

Each $W_c$ is infinite. Indeed, if $W_c$ were finite, then
\[
  R_c
  \subseteq
  \bigcup_{y\in W_c}f^{-1}(y),
\]
which would be finite, since $W_c$ is finite and every fibre of
$f$ is finite. This contradicts the infinitude of $R_c$.

Every infinite c.e.\ set contains an infinite computable subset.
For completeness, fix $c<2$ and an enumeration of $W_c$.
Choose $w_c(0)$ by waiting for an element to be enumerated.
After $w_c(n)$ has been chosen, wait for an enumerated element
greater than $w_c(n)$ and choose it as $w_c(n+1)$.
Each search terminates because $W_c$ is infinite. Thus
\[
  w_c(0)<w_c(1)<\cdots
\]
is a total computable sequence of elements of $W_c$.

Its range
\[
  S_c=\{w_c(n):n\in\N\}
\]
is infinite and computable. Indeed, strict increase implies
$w_c(n)\geq n$, so
\[
  x\in S_c
  \iff
  (\exists n\leq x)\,[w_c(n)=x],
\]
which is a decidable condition. By construction, $S_c\subseteq W_c$.

Consequently,
\[
  S_1\subseteq B
  \qquad\text{and}\qquad
  S_0\subseteq\overline B,
\]
so $B$ has computable reservoirs of both colours.
\end{proof}

We next record the coarse equivalence shared by all finite multiples
along a weak dyadic tower. Although these sets will represent distinct
one-one degrees once noncylindricity is imposed, they all remain inside
the same bounded finite-one degree.

\begin{lemma}[Bounded equivalence of nonzero tower multiples]
\label{lem:scalar-bfo}

Let $(U_n)$ be a weak dyadic tower over $U=U_0$. Then, for every
$n\in\N$ and every $m\geq1$,
\[
  mU_n\equiv_{\bfo}U.
\]
\end{lemma}

\begin{proof}
First observe that one-one equivalence is preserved under finite
autojoins. Indeed, if $f:A\leq_1B$ and $r\geq1$, then
\[
  \langle i,x\rangle_r
  \longmapsto
  \langle i,f(x)\rangle_r
\]
is a one-one reduction of $rA$ to $rB$.

We claim that, for every $n\in\N$,
\[
  U\equiv_1 2^nU_n.
\]
For $n=0$, this follows from $U_0=U$ and the computable
bijection $x\mapsto\langle 0,x\rangle_1$ witnessing
$U_0\equiv_1 1U_0$. Suppose that
\[
  U\equiv_1 2^nU_n.
\]
Since $U_n\equiv_12U_{n+1}$, applying the witnessing reductions
on each of the $2^n$ columns gives
\[
  2^nU_n\equiv_1 2^n(2U_{n+1})
  \equiv_1 2^{n+1}U_{n+1},
\]
where the last equivalence is induced by a computable permutation
of the columns. This proves the claim by induction. Notice that
only finitely many tower equivalences are used for each fixed $n$,
so no uniformity in the tower is required.

Now fix $r\geq1$. The map
\[
  x\longmapsto\langle0,x\rangle_r
\]
is a one-one reduction of $U_n$ to $rU_n$. Conversely, under the
fixed computable coding of
$\{0,\ldots,r-1\}\times\N$, the projection
\[
  \langle i,x\rangle_r\longmapsto x
\]
is a bounded finite-one reduction of $rU_n$ to $U_n$, with every
fibre having cardinality exactly $r$. Hence
\[
  rU_n\equiv_{\bfo}U_n
  \qquad(r\geq1).
\]

Taking $r=2^n$ and using
\[
  U\equiv_12^nU_n,
\]
we obtain
\[
  U\equiv_{\bfo}U_n.
\]
Taking instead $r=m$ gives
\[
  mU_n\equiv_{\bfo}U_n.
\]
By transitivity of bounded finite-one reducibility,
\[
  mU_n\equiv_{\bfo}U,
\]
as required.
\end{proof}

The dyadic tower now admits a natural numerical parametrization.
Passing from $U_n$ to $U_{n+1}$ divides the corresponding scale by
two, while taking a finite autojoin multiplies it by an integer.
Noncylindricity ensures that these numerical values faithfully record
one-one reducibility, and binary join will correspond to addition.

Let
\[
  \D_{>0}=\left\{\frac{m}{2^n}:m\geq1,\ n\in\N\right\}
\]
be the set of positive dyadic rationals.

\begin{lemma}[Scalar calculus]\label{lem:scalar-calculus}
Let $(U_n)$ be a weak dyadic tower over a noncylindrical set $U=U_0$.
For each $q\in\D_{>0}$, choose a representation
$q=m/2^n$ with $m\geq1$ and $n\in\N$, and put
\[
  S_q=mU_n.
\]
Then the one-one degree of $S_q$ is independent of the chosen
representation, and, for all $q,r\in\D_{>0}$,
\begin{equation}\label{eq:scalar-order}
  S_q\leq_1S_r \iff q\leq r,
\end{equation}
\begin{equation}\label{eq:scalar-addition}
  S_q\join S_r\equiv_1S_{q+r}.
\end{equation}
\end{lemma}

\begin{proof}
First observe that finite autojoins preserve one-one equivalence.
Indeed, if $f:A\leq_1B$ and $a\geq1$, then
\[
  \langle i,x\rangle_a
  \longmapsto
  \langle i,f(x)\rangle_a
\]
is a one-one reduction of $aA$ to $aB$.

Consequently, if $N\geq n$, then iterating the finitely many tower
equivalences between levels $n$ and $N$ gives
\[
  mU_n\equiv_1m2^{N-n}U_N.
\]
Only finitely many witnesses are used for each fixed $n$ and $N$,
so no uniformity in the tower is required.

We next show that every $U_N$ is noncylindrical. Cylinderhood is
invariant under one-one equivalence. Indeed, if $A\equiv_1B$ and
$\Cyl(B)\leq_1B$, then a one-one reduction $f:A\leq_1B$ induces
\[
  \Cyl(A)\leq_1\Cyl(B)
\]
by
\[
  \langle x,t\rangle\longmapsto\langle f(x),t\rangle.
\]
Hence
\[
  \Cyl(A)\leq_1\Cyl(B)\leq_1B\leq_1A.
\]
Since the canonical map
\[
  x\longmapsto\langle x,0\rangle
\]
always gives $A\leq_1\Cyl(A)$, the set $A$ is a cylinder.

Suppose now that $U_N$ were a cylinder. For every $a\geq1$, the map
\[
  \langle i,x\rangle_a
  \longmapsto
  \langle x,i\rangle
\]
is a one-one reduction of $aU_N$ to $\Cyl(U_N)$. Thus
\[
  aU_N\leq_1\Cyl(U_N)\leq_1U_N,
\]
while the inclusion into one column gives
\[
  U_N\leq_1aU_N.
\]
Therefore $aU_N\equiv_1U_N$ for every $a\geq1$. In particular,
\[
  2^NU_N\equiv_1U_N.
\]
Since the tower gives
\[
  U\equiv_12^NU_N,
\]
we would have $U\equiv_1U_N$, and the invariance of cylinderhood
would make $U$ a cylinder, a contradiction. Hence every $U_N$ is
noncylindrical.

We now prove that the degree of $S_q$ is independent of the chosen
representation of $q$. Suppose
\[
  q=\frac{m}{2^n}=\frac{m'}{2^{n'}}.
\]
Choose $N\geq n,n'$. Then
\[
  m2^{N-n}=m'2^{N-n'},
\]
and hence
\[
  mU_n
  \equiv_1m2^{N-n}U_N
  =
  m'2^{N-n'}U_N
  \equiv_1m'U_{n'}.
\]
Thus the one-one degree of $S_q$ depends only on $q$.

Next let
\[
  q=\frac{m}{2^n},
  \qquad
  r=\frac{k}{2^\ell},
\]
choose $N\geq n,\ell$, and put
\[
  a=m2^{N-n},
  \qquad
  b=k2^{N-\ell}.
\]
Then
\begin{equation}\label{eq:common-level-representation}
  S_q\equiv_1aU_N,
  \qquad
  S_r\equiv_1bU_N.
\end{equation}
Moreover,
\[
  q\leq r\iff a\leq b.
\]

Since $U_N$ is noncylindrical, Lemma~\ref{lem:strict-autojoins}
implies
\[
  aU_N\leq_1bU_N
  \iff
  a\leq b.
\]

Indeed, if $a\leq b$, inclusion of columns gives
$aU_N\leq_1bU_N$.

If $a>b$ and $aU_N\leq_1bU_N$, then
\[
  (b+1)U_N\leq_1aU_N\leq_1bU_N,
\]
contradicting
\[
  bU_N<_1(b+1)U_N.
\]
Together with (\ref{eq:common-level-representation}), this proves
\[
  S_q\leq_1S_r\iff q\leq r.
\]
One-one equivalence is also preserved under binary joins,
by applying the corresponding reductions separately on the
two columns.
Finally,
\[
\begin{aligned}
  S_q\join S_r
  &\equiv_1 aU_N\join bU_N\\
  &\equiv_1(a+b)U_N,
\end{aligned}
\]
where the last equivalence is induced by a computable permutation of
the columns. Since
\[
  q+r=\frac{a+b}{2^N},
\]
the well-definedness already proved gives
\[
  (a+b)U_N\equiv_1S_{q+r}.
\]
Therefore
\[
  S_q\join S_r\equiv_1S_{q+r}.
\]
\end{proof}

We can now combine the preceding ingredients. The base absorption
property initially controls only joins $U\join X$ with $X\leq_1U$;
cancellation will propagate this control throughout the dyadic tower.
An arbitrary bounded finite-one reduction can then be decomposed into
finitely many injective pieces, allowing the scalar calculus to be
applied successively. The resulting scalar increment will finally
identify the one-one degree of the original set.

\begin{theorem}[Absorption-to-Exhaustivity]
\label{thm:absorption-exhaustivity}
Let $U=U_0,U_1,\ldots$ satisfy the following conditions:
\begin{enumerate}[label=\textup{(\roman*)}]
\item $(U_n)$ is a weak dyadic tower;
\item $U$ is not a cylinder;
\item $U$ has computable reservoirs of both colours;
\item $(U,(U_n))$ has the base absorption property.
\end{enumerate}
Then every $B\equiv_{\bfo}U$ is one-one equivalent to $mU_n$ for some
$m\geq1$ and $n\in\N$. Consequently,
\begin{equation}\label{eq:degree-order-isomorphism}
  \left(
    \{[B]_1:B\equiv_{\bfo}U\},\leq_1
  \right)
  \cong
  (\D_{>0},\leq).
\end{equation}
\end{theorem}

\begin{proof}
We divide the proof into four steps.
Throughout the proof, saying that a set is a scalar means that
it is one-one equivalent to $S_q$ for some $q\in\D_{>0}$.
Recall that
\[
  U\equiv_1 S_1
  \qquad\text{and}\qquad
  U_n\equiv_1 S_{2^{-n}}
  \quad(n\in\N).
\]

\medskip
\noindent
\textbf{Step 1: absorption propagates from $U$ to every scalar.}
Let $X\leq_1U$. By base absorption,
\begin{equation}\label{eq:base-absorption-scalar}
  U\join X\equiv_1S_q
\end{equation}
for some $q\in\D_{>0}$. Since $U\leq_1U\join X$, scalar calculus gives
$q\geq1$.

Fix $n>0$ and put
\[
  Z_n=(2^n-1)U_n.
\]
Then
\[
  U\equiv_1Z_n\join U_n,
\]
and $Z_n$ has scalar value
\[
  z_n=1-2^{-n}.
\]
From (\ref{eq:base-absorption-scalar}),
\begin{equation}\label{eq:propagation-cancellation-setup}
  Z_n\join(U_n\join X)
  \equiv_1
  S_q
  \equiv_1
  Z_n\join S_{q-z_n}.
\end{equation}

Here $q-z_n\in\D_{>0}$, since $q\geq1$ and
$z_n=1-2^{-n}$. After pre- and post-composing the forward
reduction in (\ref{eq:propagation-cancellation-setup}) with the
canonical computable permutations witnessing associativity and
commutativity of finite joins, we obtain
\[
  (U_n\join X)\join Z_n
  \leq_1
  S_{q-z_n}\join Z_n.
\]
Apply Lemma~\ref{lem:cancellation} with
\[
  P=U_n\join X,\qquad Q=S_{q-z_n},\qquad Z=Z_n.
\]
Since $n>0$, the set
\[
   Z_n=(2^n-1)U_n
\]
is a nonzero scalar, and since $q-z_n>0$, the set
$S_{q-z_n}$ is also a nonzero scalar. By
Lemma~\ref{lem:scalar-bfo},
\[
   Z_n\leq_{\bfo}U\leq_{\bfo}S_{q-z_n}.
\]
By transitivity,
\[
   Z_n\leq_{\bfo}S_{q-z_n},
\]
and hence
\[
   Z_n\leq_{\fo}S_{q-z_n}.
\]
Lemma~\ref{lem:cancellation} therefore gives
\[
   U_n\join X\leq_1S_{q-z_n}.
\]

For the reverse direction, after applying the corresponding
computable permutations to the reverse reduction in
(\ref{eq:propagation-cancellation-setup}), we obtain
\[
  S_{q-z_n}\join Z_n
  \leq_1
  (U_n\join X)\join Z_n.
\]
Apply Lemma~\ref{lem:cancellation} with
\[
  P=S_{q-z_n},\qquad Q=U_n\join X,\qquad Z=Z_n.
\]
The projection gives
\[
  Z_n\leq_{\fo}U_n,
\]
and the canonical inclusion gives
\[
  U_n\leq_1U_n\join X.
\]
Hence, by transitivity,
\[
  Z_n\leq_{\fo}U_n\join X.
\]
A second application of cancellation yields
\[
  S_{q-z_n}\leq_1U_n\join X.
\]

Therefore
\begin{equation}\label{eq:propagated-absorption}
  U_n\join X\equiv_1S_{q-z_n}.
\end{equation}

For $n=0$, the conclusion
\[
  U_0\join X\equiv_1S_q
\]
is exactly (\ref{eq:base-absorption-scalar}). Therefore, for every
$n\in\N$, the set $U_n\join X$ is a scalar.

It follows that, for every scalar $S_r=mU_n$ and every
$X\leq_1U$, the set $S_r\join X$ is a scalar. If $m=1$, use the
preceding conclusion. If $m>1$, write, up to a computable
permutation of the summands,
\[
  mU_n\join X
  \equiv_1
  (m-1)U_n\join(U_n\join X),
\]
and apply scalar addition.

\medskip
\noindent
\textbf{Step 2: finite-rank decomposition of a bounded reduction.}
Let $B\equiv_{\bfo}U$, and fix a reduction
\[
  f:B\leq_{\bfo}U
\]
whose fibres have size at most some $k\geq1$. Define
\begin{equation}\label{eq:rank-partition}
  D_i=\{x:|\{z<x:f(z)=f(x)\}|=i\},
  \qquad i<k.
\end{equation}
The sets $D_i$ are computable, form a partition of $\N$, and
$f\upharpoonright D_i$ is injective.

We now make a finite nonuniform choice. Fix
\[
  I=\{i<k:D_i\text{ is infinite}\}
\]
and let
\[
  F=\bigcup_{i\notin I}D_i.
\]
Since the $D_i$ form a finite partition of $\N$, the set $I$ is nonempty
and $F$ is finite. The set $I$, the finite set $F$, and the finitely many
membership values of $B$ on $F$ need not be obtained uniformly from an
index of $f$; they are now fixed and may be hardcoded into the reductions
below.

For each $i\in I$, let $d_i:\N\to D_i$ be the computable increasing
bijection and put
\[
  X_i=\{x:d_i(x)\in B\}.
\]
Then $X_i\leq_1U$, witnessed by $f\circ d_i$. Since
$U\leq_{\bfo}B$, Lemma~\ref{lem:transfer-reservoirs} shows that $B$
has reservoirs of both colours.

Write
\[
  I=\{i_0<i_1<\cdots<i_{\ell-1}\},
\]
where $\ell=|I|\geq1$, and put
\[
  J=\bigoplus_{a<\ell}X_{i_a}
  :=
  \{\langle a,x\rangle_\ell:
       a<\ell\ \&\ x\in X_{i_a}\}.
\]
Define
\[
  \pi_J:\N\longrightarrow\N\setminus F
\]
by
\[
  \pi_J(\langle a,x\rangle_\ell)=d_{i_a}(x)
  \qquad(a<\ell,\ x\in\N).
\]
Since the sets $D_{i_a}$ are pairwise disjoint and
\[
  \N\setminus F=\bigcup_{a<\ell}D_{i_a},
\]
the map $\pi_J$ is a computable bijection from $\N$ onto
$\N\setminus F$. Moreover,
\[
\begin{aligned}
  \langle a,x\rangle_\ell\in J
  &\iff x\in X_{i_a}\\
  &\iff d_{i_a}(x)\in B\\
  &\iff \pi_J(\langle a,x\rangle_\ell)\in B.
\end{aligned}
\]

Removing $F$ from the two reservoirs of $B$ and taking their
inverse images under $\pi_J$ gives computable reservoirs of both
colours in $J$. The set
\[
  D=\N\setminus F
\]
is computable, and
\[
  B\cap\overline D=B\cap F
\]
is computable from the hardcoded membership values. Apply
Lemma~\ref{lem:remainder} with
\[
  A=J,\qquad B=B,\qquad D=\N\setminus F,\qquad \pi=\pi_J.
\]
We obtain
\begin{equation}\label{eq:bounded-rank-decomposition}
  B\equiv_1J=\bigoplus_{a<\ell}X_{i_a}.
\end{equation}

Fix $q\in\D_{>0}$. Starting with $S_q$, apply Step~1
successively to
\[
  X_{i_0},\ldots,X_{i_{\ell-1}}.
\]

Thus $S_q\join J$ is a scalar. Since
$B\equiv_1J$, it follows that $S_q\join B$ is a scalar as well.
Therefore there is a unique $T_B(q)\in\D_{>0}$ such that
\begin{equation}\label{eq:translation-map-definition}
  S_q\join B\equiv_1S_{T_B(q)}.
\end{equation}
Uniqueness follows from Lemma~\ref{lem:scalar-calculus}.

\medskip
\noindent
\textbf{Step 3: the increment is independent of the initial scalar.}
For $q,r\in\D_{>0}$,
\[
\begin{aligned}
  S_{q+r}\join B
  &\equiv_1 S_q\join(S_r\join B)\\
  &\equiv_1 S_q\join S_{T_B(r)}\\
  &\equiv_1 S_{q+T_B(r)}.
\end{aligned}
\]
Thus
\begin{equation}\label{eq:translation-law}
  T_B(q+r)=q+T_B(r).
\end{equation}
Interchanging $q$ and $r$ gives
\[
  T_B(q)-q=T_B(r)-r.
\]

Hence there is a dyadic rational $\delta(B)$ such that
\begin{equation}\label{eq:constant-increment}
  T_B(q)=q+\delta(B)
  \qquad(q\in\D_{>0}).
\end{equation}
Since $S_q\leq_1S_q\join B$, scalar calculus gives
$T_B(q)\geq q$, and hence $\delta(B)\geq0$.

The increment is in fact positive.
Choose a bounded finite-one reduction
$g:U\leq_{\bfo}B$ with fibre bound $h\geq1$.
Define
\[
  r_g(x)=|\{z<x:g(z)=g(x)\}|.
\]
Then $r_g(x)<h$, and the map
\[
  x\longmapsto\langle r_g(x),g(x)\rangle_h
\]
is a total computable injective reduction of $U$ to $hB$. Thus
\begin{equation}\label{eq:rank-tagged-embedding}
  U\leq_1hB.
\end{equation}
Therefore
\[
  2U\leq_1U\join hB.
\]
We claim that, for every $j\geq1$ and every $q\in\D_{>0}$,

\begin{equation}\label{eq:iterated-increment}
   S_q\join jB\equiv_1S_{q+j\delta(B)}.
\end{equation}

For $j=1$, this follows from (\ref{eq:translation-map-definition})
and (\ref{eq:constant-increment}). Suppose that the
claim holds for some $j\geq1$. Then, using associativity of the
join, the induction hypothesis, and again (\ref{eq:translation-map-definition})
and (\ref{eq:constant-increment}), we have
\[
\begin{aligned}
   S_q\join (j+1)B
   &\equiv_1 (S_q\join jB)\join B\\
   &\equiv_1 S_{q+j\delta(B)}\join B\\
   &\equiv_1 S_{q+(j+1)\delta(B)}.
\end{aligned}
\]
Here $q+j\delta(B)\in\D_{>0}$, since $q\in\D_{>0}$ and $\delta(B)\geq0$ is dyadic.
This proves the claim by induction.
Taking $q=1$ and $j=h$, we obtain
\[
  U\join hB\equiv_1S_{1+h\delta(B)}.
\]
Scalar calculus now yields
\[
  2\leq1+h\delta(B),
\]
so
\begin{equation}\label{eq:positive-increment}
  \delta(B)\geq\frac1h>0.
\end{equation}

\medskip
\noindent
\textbf{Step 4: cancellation recovers $B$.}
By (\ref{eq:constant-increment}),
\begin{equation}\label{eq:final-common-summand}
  U\join B
  \equiv_1
  S_{1+\delta(B)}
  \equiv_1
  U\join S_{\delta(B)}.
\end{equation}

After computably interchanging the two summands on both sides of
the forward reduction in (\ref{eq:final-common-summand}), we obtain
\[
  B\join U
  \leq_1
  S_{\delta(B)}\join U.
\]
Apply Lemma~\ref{lem:cancellation} with
\[
  P=B,\qquad Q=S_{\delta(B)},\qquad Z=U.
\]
By Lemma~\ref{lem:scalar-bfo}, $U\leq_{\bfo}S_{\delta(B)}$, and hence
$U\leq_{\fo}S_{\delta(B)}$. We obtain
\[
  B\leq_1S_{\delta(B)}.
\]

For the reverse direction, after computably interchanging the
summands on both sides of the reverse reduction in
(\ref{eq:final-common-summand}), we obtain
\[
  S_{\delta(B)}\join U
  \leq_1
  B\join U.
\]
Apply Lemma~\ref{lem:cancellation} with
\[
  P=S_{\delta(B)},\qquad Q=B,\qquad Z=U.
\]
Since $B\equiv_{\bfo}U$, there is a bounded finite-one reduction
\[
  U\leq_{\bfo}B,
\]
which is in particular finite-one. Thus
\[
  U\leq_{\fo}B,
\]
and the lemma gives
\[
  S_{\delta(B)}\leq_1B.
\]

Thus
\[
  B\equiv_1S_{\delta(B)}=mU_n
\]
for suitable $m\geq1$ and $n$.

Every $mU_n$ is bounded finite-one equivalent to $U$, so the map
\[
  \frac{m}{2^n}\longmapsto[mU_n]_1
\]
is onto the set of one-one degrees in the bounded finite-one degree of
$U$. By Lemma~\ref{lem:scalar-calculus}, it is well defined, injective,
and order preserving in both directions. This proves (\ref{eq:degree-order-isomorphism}).
\end{proof}

\section{The dense degree}

The two preceding theorems now fit together directly. The block
construction supplies a set satisfying the hypotheses of the abstract
exhaustivity theorem, while the scalar calculus identifies the resulting
one-one degrees with the positive dyadic rationals. It remains only to
observe that this order is the countable dense linear order without
endpoints.

\begin{corollary}\label{cor:main}
There exists a bounded finite-one degree such that the one-one
degrees contained in it, ordered by $\leq_1$, form a countable
dense linear order without endpoints. More precisely, there is
a set $U\leq_T\emptyset''$ such that

\[
  \left(
    \{[B]_1:B\equiv_{\bfo}U\},\leq_1
  \right)
  \cong
  (\D_{>0},\leq)
  \cong
  (\mathbb Q,\leq).
\]
\end{corollary}

\begin{proof}

Let $U\leq_T\emptyset''$ and $(U_e)_{e\in\N}$ be supplied by
Theorem~\ref{thm:block-homogenization}.
Then $(U_e)_{e\in\N}$
is a weak dyadic tower over $U$, the set $U$ is noncylindrical and
has computable reservoirs of both colours, and
$(U,(U_e)_{e\in\N})$ has the base absorption property. Hence
Theorem~\ref{thm:absorption-exhaustivity} gives
\[
  \left(
    \{[B]_1:B\equiv_{\bfo}U\},\leq_1
  \right)
  \cong
  (\D_{>0},\leq).
\]

The set $\D_{>0}$ is countable and linearly ordered by the usual
order. If $q,r\in\D_{>0}$ and $q<r$, then
\[
  q<\frac{q+r}{2}<r,
\]
and $(q+r)/2\in\D_{>0}$, so the order is dense. Moreover, for every
$q\in\D_{>0}$,
\[
  0<\frac q2<q<2q,
\]
so the order has neither a least nor a greatest element. By Cantor's
characterization of the countable dense linear orders without
endpoints,
\[
  (\D_{>0},\leq)\cong(\mathbb Q,\leq).
\]
\end{proof}

\begin{remark}

The preceding corollary combines two deliberately separated parts of the
argument. The global construction and its effectivity analysis relative to
$\emptyset''$ are confined to Theorem~\ref{thm:block-homogenization} and its
preparatory lemmas.

By contrast, Theorem~\ref{thm:absorption-exhaustivity} is an
abstract reduction-theoretic argument using only a weak dyadic
tower, noncylindricity, computable reservoirs, base absorption,
scalar calculus, finite-rank decomposition, and finite-one
cancellation. Its proof includes finite nonuniform choices and
effective constructions of reductions, but every resulting
reduction is an ordinary computable function. Whenever tower
witnesses are used, only finitely many levels, fixed for the
reduction as a whole, are involved.

No oracle construction or computable uniform family of tower
witnesses is required. Nor is a uniform procedure for selecting
all the finite nonuniform data asserted. This separation is
intended to facilitate independent verification of the proof.
\end{remark}

\section{Conclusion}

We have answered Open Question~3 of Richter, Stephan, and Zhang
\cite{RSZ2026} affirmatively. For a noncomputable set
$U\leq_T\emptyset''$, the one-one degrees in $[U]_{\bfo}$ are
exactly those represented by the finite autojoins $mU_n$ of
the levels of a weak dyadic tower, ordered according to the
positive dyadic values $m/2^n$. The full internal order is
therefore isomorphic to $(\mathbb Q,\leq)$, rather than merely
containing a copy of it. Together with the $m$-rigidity obstruction
in \cite{CintioliRigidity}, this shows that the class of sets whose
bounded finite-one degree has this exact internal order type is
nonempty but null and meager in Cantor space.

A further question suggested by the present construction is
whether such an exact dense realization can be obtained with
a representative $U\leq_T\emptyset'$. It would also be
interesting to develop versions of the absorption-to-exhaustivity
principle for other families of representatives, with the aim
of obtaining further exact realizations.

\section*{Acknowledgements}

The main result of this paper was discovered by ChatGPT 5.6 Sol Pro (OpenAI).

The author has verified the arguments and bears sole responsibility for the correctness of all statements, proofs, and conclusions.

\end{document}